\documentclass[letterpaper, 10 pt, conference]{ieeeconf}  

\usepackage{amssymb,latexsym,amsmath}
\usepackage{epsfig}
\usepackage{caption}

\let\labelindent\relax

\def\Pr{\mathop{\rm Pr}}

\def\argmin{\mathop{\rm arg\, min}}

\def\B{{\mathcal B}}

\def\P{{\mathcal P}}

\usepackage{amsmath}
\usepackage{amsfonts}
\usepackage{array}
\usepackage{dsfont}
\usepackage{hyperref}
\usepackage{amssymb}

\usepackage{amsthm}
\usepackage{bbold}
\usepackage{caption}
\usepackage{standalone}
\usepackage{accents}
\usepackage{enumitem}
\usepackage{tikz}
\usepackage{pgfplots}
\pgfplotsset{compat=1.6}
\usepgfplotslibrary{groupplots}
\usepackage{cancel}
\allowdisplaybreaks

\newtheorem{lemma}{Lemma}
\newtheorem{theorem}{Theorem}
\newtheorem{prop}{Proposition}
\newtheorem{assumption}{Assumption}

\theoremstyle{definition}

\theoremstyle{remark}
\newtheorem*{remark}{Remark}

\usepackage{amssymb,latexsym,amsmath}
\usepackage{mathrsfs}
\usepackage{epsfig}
\usepackage{caption}

\newcommand{\R}{\mathds{R}}
\newcommand{\Zplus}{\mathds{Z}_+}

\pgfplotsset{soldot/.style={color=blue,only marks,mark=*}}
\pgfplotsset{holdot/.style={color=blue,fill=white,only marks,mark=*}}

\fussy

\IEEEoverridecommandlockouts

\renewcommand{\baselinestretch}{0.94}

\allowdisplaybreaks

\begin{document}
\sloppy
\title{ Learning POMDPs with Linear Function Approximation and Finite Memory
\thanks{}
}

\author{Ali Devran Kara
\thanks{The author is with the Department of Mathematics,
     Florida State University, Tallahassee, FL, USA,
     Email: akara@fsu.edu}
     }
\maketitle

\begin{abstract}
We study reinforcement learning with linear function approximation and finite-memory approximations for partially observed Markov decision processes (POMDPs). We first present an algorithm for the value evaluation of finite-memory feedback policies. We provide error bounds derived from filter stability and projection errors. We then study the learning of  finite-memory based near-optimal Q values. Convergence in this case requires further assumptions on the exploration policy when using general basis functions. We then show that these assumptions can be relaxed for specific models such as those with perfectly linear cost and dynamics, or when using discretization based basis functions.
\end{abstract}

\section{Introduction}
It is well known that for the optimality analysis and learning of stochastic control problems, curse of dimensionality is one of the main challenges. Function approximation methods are widely used to tackle this issue. In particular, reinforcement learning based on linear functions approximations studied in depth in the literature for fully observed control problems. 

\cite{tsitsiklis1997analysis} is among the first studies for the analysis of linear function approximation for policy evaluation in fully observed MDPs. The authors show the convergence of TD($\lambda$) methods under linear function approximation. However, the analysis for learning the optimal Q-values under linear function approximation is more challenging. In particular, the invariant measure of the exploration policy, which is used for projection onto the span of basis functions, does not align with the measure induced by the greedy action selection policy.  Hence, the algorithm may not converge in general except for particular cases.  \cite{melo2008analysis} shows the convergence under a restrictive assumption which suggests that the exploration policy should be close to the greedy action selection policy. Another special case that guarantees convergence is the exact representation case. If the optimal Q value is perfectly linear, i.e. belongs to the span of the basis functions, then it can be learned exactly. In particular, the composition operator that involves the projection mapping and the Bellman operator coincides with the Bellman operator itself;  hence, this composition remains a contraction under the uniform norm. \cite{Ruszczy2024,jin2023provably} consider this exact representation case. Another special case is when the basis functions are orthonormal (e.g. discretization based approximation). In this case, the projection map is not only non-expansive in the $L_2$ norm but also in the uniform norm, and thus a convergence and error analysis can be made without restrictive conditions, see \cite{kara2023q}.  For general basis functions, Meyn \cite{meyn2024projected} recently showed that even though the composition (projection+Bellman) operator is not necessarily a contraction, it does admit at least one fixed point solution if the exploration policy is $\epsilon$-greedy. Furthermore, it is shown that the parameter iterations remain bounded almost surely.

Function approximation for POMDPs remains relatively less studied. \cite{cai2022reinforcement} is one of the few that studies this topic. The authors assume that the transition kernel and the observation channel admit densities and that the densities are exactly realizable by the basis functions. They consider finite-memory variables for the learning.  A restrictive assumption on the observation model, which ensures the invertibility of the observation distributions, guarantees that the Bellman mapping for the finite-memory variables can be parametrized as well. This observability condition implies that for any distribution on the observations, one can fully recover the distribution on the hidden state variable. 

In this paper, we consider linear function approximation for functions of finite-memory variables. We do not assume exact realizability conditions. We first study policy evaluation for finite-memory feedback policies as well as  learning of optimal Q-values based on the finite-memory variables. The results are parallel to the fully observed counterparts. For policy evaluation, we show that the convergence holds under ergodicity assumptions on the finite-memory variables. We provide upper bounds on the error of the learned value, building on the finite-memory approximation framework
developed in \cite{kara2020near,kara2023convergence}. For optimal Q-value learning, we impose an additional assumption similar to the analysis in \cite{melo2008analysis}. Finally, for discretization-based basis functions, we show the convergence and the error analysis of \cite{devran2025near} apply under less restrictive assumptions on the model and the exploration policies.

\subsection{Partially Observed Markov Decision Processes}
Let $\mathds{X} \subset \mathds{R}^m$ denote a Borel set which is the state space of a POMDP for some $m\in\mathds{N}$. Let
$\mathds{Y} \subset \mathds{R}^n$ be another Borel set denoting the observation space of the model, and let the state be observed through an
observation channel $O$. The observation channel, $O$, is defined as a stochastic kernel (regular
conditional probability) from  $\mathds{X}$ to $\mathds{Y}$, such that
$O(\,\cdot\,|x)$ is a probability measure on the sigma algebra $\mathcal{B}(\mathds{Y})$ of $\mathds{Y}$ for every $x
\in \mathds{X}$, and $O(A|\,\cdot\,): \mathds{X}\to [0,1]$ is a Borel
measurable function for every $A \in \mathcal{B}(\mathds{Y})$.
$\mathds{U}\in \mathds{R}^l$ denotes the action space. An {\em admissible policy} $\gamma$ is a
sequence of control functions $\{\gamma_t,\, t\in \Zplus\}$ such
that $\gamma_t$ is measurable with respect to the $\sigma$-algebra
generated by the information variables
$
I_t=\{Y_{[0,t]},U_{[0,t-1]}\}, \quad t \in \mathds{N}, \quad
  \quad I_0=\{Y_0\},
$
where $U_t=\gamma_t(I_t),\quad t\in \Zplus$,
are the $\mathds{U}$-valued control
actions and 
$Y_{[0,t]} = \{Y_s,\, 0 \leq s \leq t \}, \quad U_{[0,t-1]} =
  \{U_s, \, 0 \leq s \leq t-1 \}.$ We define $\Gamma$ to be the set of all such admissible policies. The update rules of the system are determined by
relationships:
\[  \Pr\bigl( (X_0,Y_0)\in B \bigr) =  \int_B \mu(dx_0)O(dy_0|x_0), \quad B\in \mathcal{B}(\mathds{X}\times\mathds{Y}), \]
where $\mu$ is the (prior) distribution of the initial state $X_0$, and
\begin{eqnarray*}
\label{eq_evol}
 &\Pr\biggl( (X_t,Y_t)\in B \, \bigg|\, (X,Y,U)_{[0,t-1]}=(x,y,u)_{[0,t-1]} \biggr)\\
& = \int_B \mathcal{T}(dx_t|x_{t-1}, u_{t-1})O(dy_t|x_t),  
\end{eqnarray*}
$B\in \mathcal{B}(\mathds{X}\times\mathds{Y}), t\in \mathds{N},$ where $\mathcal{T}$ is the transition kernel of the model which is a stochastic kernel from $\mathds{X}\times
\mathds{U}$ to $\mathds{X}$.  We let the objective of the agent (decision maker) be the minimization of the infinite horizon discounted cost, 
  \begin{align}\label{criterion1}
    J_{\beta}(\mu,\gamma)= E_\mu^{\gamma}\left[\sum_{t=0}^{\infty} \beta^t c(X_t,U_t)\right]
  \end{align}
 \noindent for some discount factor $\beta \in (0,1)$, over the set of admissible policies $\gamma\in\Gamma$, where $c:\mathds{X}\times\mathds{U}\to\R$ is a Borel-measurable stage-wise cost function and $E_\mu^{\gamma}$ denotes the expectation with initial state probability measure $\mu$ and transition kernel $\mathcal{T}$ and the channel $O$ under policy $\gamma$. Note that $\mu\in\mathcal{P}(\mathds{X})$, where we let $\mathcal{P}(\mathds{X})$ denote the set of probability measures on $\mathds{X}$. We define the optimal cost for the discounted infinite horizon setup as a function of the priors as
\begin{align}\label{opt_val}
  J_{\beta}^*(\mu)&=\inf_{\gamma\in\Gamma} J_{\beta}(\mu,\gamma).
\end{align}
For the analysis of partially observed MDPs, a common approach is to reformulate the problem as a fully observed MDP where the decision maker keeps track of the posterior distribution of the state $X_t$ given the available history $I_t$, also called the belief MDP. In what follows, we will use an alternative yet related reformulation based on finite-memory (window) information variables.

\section{Reduction to Fully Observed Using finite-memory Variables}\label{finite_memory}
 The following construction  is mostly taken from \cite{kara2023convergence}, however, we present the method in detail for completeness.
\subsection{Finite-memory belief-MDP reduction}\label{alt_sec}
 We construct a fully observed MDP reduction using the predictor from $N$ stages earlier and the most recent $N$ information variables (that is, measurements and actions). 
Consider the following state variable at time $t$:
\begin{align}\label{finite_belief_state}
{z}_t=(\mu_{t-N},h_t^N)
\end{align}
where, for $N\geq 1$
\begin{align*}
\mu_{t-N}&=Pr(X_{t-N}\in \cdot|y_{t-N-1},\dots,y_0,u_{t-N-1},\dots,u_0),\\
h_t^N&=\{y_t,\dots,y_{t-N},u_{t-1},\dots,u_{t-N}\}
\end{align*}
and $h_t^N=y_t$ for $N=0$
with $\mu$ being the prior probability measure on $X_0$. The state space with this representation is ${\mathcal{Z}}=\P(\mathds{X})\times \mathds{Y}^{N+1}\times \mathds{U}^{N}$ where we equip ${\mathcal{Z}}$ with the product topology where we consider the weak convergence topology on the $\P(\mathds{X})$ and the usual (coordinate) topologies on $\mathds{Y}^{N+1}\times \mathds{U}^{N}$. 

We can now define the stage-wise cost function and the transition probabilities. Consider the new cost function $\hat{c}:{\mathcal{Z}}\times \mathds{U}\to \mathds{R}$,  
\begin{align}\label{hat_cost}
&\hat{c}({z}_t,u_t)=\hat{c}(\mu_{t-N},h_t^N,u_t)\\
&=\int_\mathds{X}c(x_t,u_t)P^{\mu_{t-N}}(dx_t|y_t,\dots,y_{t-N},u_{t-1},\dots,u_{t-N}).\nonumber
\end{align}
Furthermore, we can define the transition probabilities for $N=1$ (for simplicity) as follows: for some $A\in \B(\mathcal{Z})$ such that 
\[\ A= B\times\{\hat{y}_{t-N+1},\hat{u}_{t},\dots,\hat{u}_{t-N+1}\},\quad  B\in\B(\P(\mathds{X})) \]
we write
\begin{align*}
&Pr({z}_{t+1}\in A|{z}_{t},\dots,{z}_0,u_{t},\dots,u_0) \\
&=Pr(\mu_t\in B,\hat{y}_{t+1},\hat{y}_{t},\hat{u}_{t}|\mu_{[t-1,0]},y_{[t,0]},u_{[t,0]})\\
&=\mathds{1}_{\{y_{t},u_{t}=\hat{y}_{t},\hat{u}_{t},G(\mu_{t-1},y_{t-1},u_{t-1})\in B\}}\\
&\qquad\qquad P^{\mu_{t-1}}(\hat{y}_{t+1}|y_t,y_{t-1},u_{t},u_{t-1})\\
&=Pr(\mu_{t}\in B,\hat{y}_{t+1},\hat{y}_{t},\hat{u}_{t}|\mu_{t-1},y_t,y_{t-1},u_t,u_{t-1})\\
&=Pr({z}_{t+1}\in A|{z}_t,u_t)=:\int_A{\eta}(d{z}_{t+1}|{z}_t,u_t)
\end{align*}
where the map $G$ is defined as 
\begin{align*}
&G(\mu_{t-1},y_{t-1},u_{t-1})=P^\mu(X_{t}\in\cdot|y_{t-1},\dots,y_0,u_{t-1},\dots,u_0).
\end{align*}
For some admissible policy $\gamma$, and some initial state $z_0\in\mathcal{Z}$ we write its induced cost as 
\begin{align*}
\hat{J}_\beta(z_0,\gamma) = \sum_{t=0}^\infty \beta^t E^\gamma[\hat{c}(Z_t,U_t)].
\end{align*}
Respectively, we denote the optimal value function by $\hat{J}^*_\beta(z_0)$. Note that this construction is without loss of optimality. In particular, for a fixed $\mu_{-N}$, assuming some arbitrary policy $\gamma$ acts from time $-N$ through $-1$, one can then show that 
\begin{align*}
E\left[ \hat{J}_\beta^*(Z_0)\right] = E\left[ \hat{J}_\beta^*(\mu_{-N} , H_0)\right] = E[J_\beta^*(\mu_0)]
\end{align*}
where the expectation on the left is with respect to $H_0=\{Y_0,\dots,Y_{-N},U_{-1},\dots, U_{-N} \}$, and on the right with respect to $\mu_{0}=Pr(X_{0}\in \cdot|Y_{-1},\dots,Y_{-N},U_{-1},\dots,U_{-N})$. Note that $J_\beta^*(\mu_0)$ is the optimal value function defined in (\ref{opt_val}).

Hence, we have a fully observed MDP, with the cost function $\hat{c}$, transition kernel ${\eta}$ and the state space ${\mathcal{Z}}$.

\subsection{Approximation of the finite-memory belief-MDP}\label{app_section}
The finite-memory belief MDP model constructed in the previous section lives in the state space
\begin{align*}
{\mathcal{Z}}=\bigg\{&\pi,y_{[0,N]},u_{[0,N-1]}:\\
&\pi\in\mathcal{P}(\mathds{X}), y_{[0,N]}\in{\mathds{Y}}^{N+1}, u_{[0,N-1]}\in\mathds{U}^N\bigg\},
\end{align*}
where the first coordinate summarizes the past information, and the second and the last coordinates carry the information from the most recent $N$ time steps.

Consider the following set ${\mathcal{Z}}_{\pi}$ for a fixed $\pi\in\P(\mathds{X})$
\begin{align*}
{\mathcal{Z}}_{\pi}=\bigg\{\pi,y_{[0,N]},u_{[0,N-1]}: y_{[0,N]}\in{\mathds{Y}}^{N+1}, u_{[0,N-1]}\in\mathds{U}^N\bigg\}
\end{align*}
such that the state at time $t$ is $\hat{z}_t=(\pi,h_t^N)$. Compared to the state ${z}_t=(\mu_{t-N},h_t^N)$ defined in (\ref{finite_belief_state}), this approximate model uses $\pi$ as the predictor, no matter what the real predictor at time $t-N$ is.

The cost function is defined as 
\begin{align}\label{cost_window}
&\hat{c}_\pi(\hat{z}_t,u_t)=\hat{c}(\pi,h_t^N,u_t)\\
&=\int_\mathds{X}c(x_t,u_t)P^{\pi}(dx_t|y_t,\dots,y_{t-N},u_{t-1},\dots,u_{t-N}).\nonumber
\end{align}
We define the controlled transition model by for some $\hat{z}_{t+1}=(\pi,h_{t+1}^N)$ and $\hat{z}_{t}=(\pi,h_{t}^N)$
\begin{align}\label{eta_N}
&{\eta}_\pi(\hat{z}_{t+1}|\hat{z}_t,u_t)={\eta}_\pi(\pi,h_{t+1}^N|\pi,h_t^N,u_t)\nonumber\\
&:={\eta}\bigg(\P(\mathds{X}),h_{t+1}^N|\pi,h_t^N,u_t\bigg).
\end{align}

For simplicity, if we assume $N=1$, then the transitions can be rewritten for some $h_{t+1}^N=(\hat{y}_{t+1},\hat{y}_t,\hat{u}_t)$ and $h_t^N=(y_t,y_{t-1},u_{t-1})$
\begin{align}\label{alt_eta_N}
&{\eta}_\pi(\pi,\hat{y}_{t+1},\hat{y}_t,\hat{u}_t|\pi,y_t,y_{t-1},u_{t-1},u_t)\nonumber\\
&={\eta}(\P(\mathds{X}),\hat{y}_{t+1},\hat{y}_t,\hat{u}_t|\pi,y_t,y_{t-1},u_{t-1},u_t)\nonumber\\
&=\mathds{1}_{\{y_t=\hat{y}_t,u_t=\hat{u}_t\}}P^{\pi}(\hat{y}_{t+1}|y_t,y_{t-1},u_t,u_{t-1}).
\end{align} 
In particular, if we have a function $f(h^N)$ that only depends on the finite-memory variables (and not on the predictor), we then have that
\begin{align}\label{consistent}
\int f(h^N_1)\eta_\pi(dh_1^N|h^N,u) = \int f(h_1^N)\eta(d\pi_1,dh_1^N|z,u)
\end{align}
where $z=(\pi,h^N)$ so that it coincides with the original dynamics.



We denote the optimal value function for the approximate model by $J_\beta^N$. 
Note that $J^N_\beta$ is defined on the set ${\mathcal{Z}}_{\pi}$. However, we can simply extend it to the set ${\mathcal{Z}}$ by defining it as constant over $\P(\mathds{X})$ for the first coordinate.  

We also note that since the predictor $\pi$ is fixed, $J_\beta^N$ can be thought as a function on $h_t^N$, the finite-memory information variables. In what follows, we sometimes directly write $J_\beta^N(h^N)$ for the value function of the approximate model defined in this section instead of writing $J_\beta^N(\pi,h^N)$.

We define the following constant:
\begin{align}\label{loss_constant}
L_t:=\sup_{\hat{\gamma}\in\hat{\Gamma}}E_{\mu_0}^{\hat{\gamma}}\bigg[\|P^{\mu_t}(X_{t+N}\in\cdot|Y_{[t,t+N]},U_{[t,t+N-1]})\nonumber\\
-P^{\pi}(X_{t+N}\in\cdot|Y_{[t,t+N]},U_{[t,t+N-1]})\|_{TV}\bigg]
\end{align} 
which is the expected value on the total variation distance between the posterior distributions of $X_{t+N}$ conditioned on the same observation and control action variables $Y_{[t,t+N]},U_{[t,t+N-1]}$ when the prior distributions of $X_{t}$ are given by $\mu_t$ and $\pi$. This filter stability term plays a significant role in the error analysis that follows. One can show that $L_t \to 0$ as $N\to 0$ (in some cases, exponentially fast) under certain assumptions. We refer the reader to \cite{kara2020near,kara2023convergence,mcdonald2020exponential} for further details on this analysis. 

\begin{prop}\cite[Theorem 3.3]{kara2023convergence}\label{cont_bound}
For ${z}_0=(\mu_0,h_0^N)$, with a policy $\hat{\gamma}$ acting on the first $N$ steps, we have that
\begin{itemize}
\item For a finite-memory policy (not necessarily optimal) $\gamma^N$
\begin{align*}
E_{\mu_0}^{\hat{\gamma}}\left[\left|J^N_\beta(h_0^N,\gamma^N) -J_\beta({z}_0,\gamma^N)\right|\right]\leq  \frac{\|c\|_\infty }{(1-\beta)}\sum_{t=0}^\infty\beta^tL_t
\end{align*}
\item For the difference between the value functions we have
\begin{align*}
&E_{\mu_0}^{\hat{\gamma}}\left[\left|J^N_\beta(h_0^N) -J^*_\beta({z}_0)\right|\right]\leq  \frac{\|c\|_\infty }{(1-\beta)}\sum_{t=0}^\infty\beta^tL_t
\end{align*}
where the expectation is with respect to the random realizations of the initial finite-memory variables $h_0^N$.
\end{itemize}
\end{prop}

\section{Policy Evaluation} In this section, we propose a finite-memory policy evaluation method using linear function approximation. We denote the finite-memory variables by $h_t$ 
by omitting the dependence on $N$. We introduce a set of basis functions $\{\phi^i(h)\}_{i=1}^d$ where  $\phi^i(h):\mathds{Y}^N\times\mathds{U}^{N-1}\to \mathds{R}$. We let $\mathds{H}:=\mathds{Y}^N\times\mathds{U}^{N-1}\to \mathds{R}$ and  we denote by ${\bf \Phi}^\intercal:= [\phi^1,\dots,\phi^d]$ the vector of the basis functions.
\begin{assumption}
We assume for the rest of the paper that $
\|\phi^i\|_\infty\leq 1$
for all $i=1,\dots,d$.
\end{assumption}

Let $\gamma(h):\mathds{H}\to \mathds{U}$ be a finite-memory feedback policy.  We assume that the unobserved state variable initiates at time $-N$, according to some $\mu_{-N}\in\P(\mathds{X})$, and the controller starts acting on time $t=0$. Recall the construction in Section \ref{alt_sec}; for some initial state variable $z_0=(\mu_{-N}, h_0)$.
In this section, our goal is to approximate the performance of the policy $\gamma$:
\begin{align*}
J_\beta(z_0,\gamma)= J_\beta(\mu_{-N},h_0,\gamma)\approx \theta^\intercal {\bf \Phi}(h_0)
\end{align*}
for some $\theta\in \mathds{R}^d$ from a single trajectory of data of observations and actions when the model is unknown.

\subsection{Ergodicity} In this part, we study the long run behavior of the process $\{h_t\}$. We note that this process is not a Markov chain. However, the joint process $(h_t,x_{t},u_t)$ is a Markov chain under a finite-memory policy $\gamma$. For example, for $N=2$ and for some $B_1, B_2 \in  \mathcal{B}(\mathds{Y}), B_3,B_4 \in \mathcal{B}(\mathds{U}), B_5 \in \mathcal{B}(\mathds{X}) $, denoting by $I_{t+1} = \{(y,x,u)_ {t+1},\dots,(y,x,u)_{0}\}$
\begin{align*}
&Pr({Y}_{t+2} \in B_1,{Y}_{t+1} \in B_2, U_{t+1} \in B_3, X_{t+2}\in B_5,\\
&\qquad\qquad\qquad\qquad\qquad\qquad\qquad\qquad   {U}_{t+2}\in B_4 | I_{t+1} )\\
&= \int_{x_{t+1}\in B_5}\int_{x_{t+2}\in\mathds{X}} \int_{y_{t+2}\in B_2} \int_{u_{t+2}\in B_4}\mathds{1}_{(y_{t+1} \in B_2, u_{t+1}\in B_3)} \\
&  \qquad\gamma(du_{t+2}|y_{t+2},y_{t+1},u_{t+1})O(dy_{t+2}|x_{t+2})\\
&\qquad\qquad\qquad  \mathcal{T}(dx_{t+2}|x_{t+1},u_{t+1})
\end{align*}
which shows that the joint process is a Markov chain. We will assume in the following analysis that under the finite-memory policy $\gamma$, this process is exponentially ergodic.
\begin{assumption}\label{erg_assmp}
Under the finite-memory policy $\gamma$, the process $(h_t,x_t,u_t)$ is exponentially ergodic.
\end{assumption}
We note that it is not possible to guarantee this assumption solely using the properties of the transition kernel $\mathcal{T}(\cdot|x,u)$  in general. This is due to the fact that the finite-memory variable $h_t$ contains the past control actions, and thus the dependence of the control policies on the past control actions makes the ergodicity analysis non-trivial.  For example, for a policy of type $u_t\sim\gamma(\cdot|u_{t-1})$, the ergodicity of the action process and thus the finite-memory process, clearly depends on the randomized policy $\gamma(\cdot|u_{t-1})$.

We  note that if the finite-memory policy $\gamma$ and the transition kernel $\mathcal{T}$ satisfy a minorization condition, then the augmented process is uniquely ergodic.
\begin{assumption}\label{minor}
We assume that there exist non-trivial measures $\lambda_x(\cdot)$ and  $\lambda_u(\cdot)$ such that 
\begin{align*}
\mathcal{T}(dx_1|x,u)&\geq\lambda_x(dx_1)\\
\gamma(du|h)&\geq \lambda_u(du)
\end{align*}
for all $(x,u)\in\mathds{X}\times\mathds{U}$ and for all $h\in \mathds{Y}^{N}\times \mathds{U^{N-1}}$.
\end{assumption}
\begin{lemma}
Assumption \ref{minor} implies Assumption \ref{erg_assmp}. In particular, under Assumption \ref{minor}, the augmented Markov chain $(h_t,x_{t},u_t)$ is exponentially ergodic under the finite-memory policy satisfying Assumption \ref{minor}. 
\end{lemma}
\begin{proof}
We give a sketch of the proof for $N=2$: consider the two step transition for the chain $(h_t,x_{t},u_t)$ for some starting point $(y_1,y_0,u_0,x_1,u_1)$:
\begin{align*}
&Pr(dy_3, dy_2,du_2,dx_3,du_3|y_1,y_0,u_0,x_1,u_1)\\
&=\int_{x_2\in\mathds{X}} \gamma(du_3|y_3,y_2,u_2)O(dy_3|x_3)\mathcal{T}(dx_3|x_2,u_2)\\
&\quad\gamma(du_2|y_2,y_1,u_1) O(dy_2|x_2) \mathcal{T}(dx_2|x_1,u_1)\\
&\geq \int_{x_2} \gamma(du_3|y_3,y_2,u_2)O(dy_3|x_3)\mathcal{T}(dx_3|x_2,u_2)\\
&\quad\lambda_u(du_2)O(dy_2|x_2) \lambda_x(dx_2)\\
&=:\lambda_h(du_3,dy_3, dy_2,du_2,dx_3)
\end{align*}
the non-trivial measure $\lambda_h(\cdot)$  is independent of the starting point, and thus it can be shown that $(h_t,x_{t},u_t)$ is exponentially ergodic (see e.g.  \cite[Lemma 3.3]{hernandez2012adaptive}.
\end{proof}
\begin{remark}
For any finite-memory policy $\gamma$ that does not satisfy Assumption \ref{minor}, one can always construct a perturbed version that does satisfy this assumption. In particular, let $\gamma'$ be an arbitrary policy that satisfies the minorizarion policy. Then, the perturbed policy $\hat{\gamma}(du|h)=(1-\epsilon)\gamma(du|h)+\epsilon \gamma'(du|h)$ satisfies Assumption \ref{minor} by construction.
\end{remark}

Let $\pi(\cdot)$ denote the invariant measure of the joint process $(h_t,x_{t},u_t)$.  We denote by $\pi_x(\cdot)$ the marginal of $\pi(\cdot)$ on $x_t$. We recall the following notation used earlier:
$
P^\mu(dx_t|h_t)
$
which denotes the Bayesian update of the distribution of $x_t$ conditioned on the finite-memory $h_t$, given that the prior measure on  $x_{t-N}$ is $\mu$. 

Through disintegration, the invariant measure $\pi(\cdot)$ on $(h_t,x_{t},u_t)$ induces a conditional distribution $Pr(dx_t|h_t)$. 
One can then show that this conditional distribution coincides with $P^{\pi_x}(dx_t|h_t)$ where the initial condition $x_{t-N}$ is distributed according to $\pi_x(\cdot)$, the marginal of the invariant measure $\pi(\cdot)$ on the $x$ variable.

\subsection{Projection and approximate Bellman mappings} In this section we introduce mappings that will be used in the upcoming analysis. Recall the invariant distribution $\pi$ of the process $(h_t,x_t,u_t)$; with an abuse of notation, we also denote its marginal on $h_t$ by $\pi$. We consider the $L_2$ space of real valued functions on $h\in \mathds{H}$ with the measure $\pi\in \P(\mathds{H})$ under the usual inner product. We denote by $\Pi^\pi$ the projection map from $ L_2(\pi,\mathds{H})$ onto the span of ${\bf \Phi}^\intercal:= [\phi^1,\dots,\phi^d]$. In particular, for some $f\in  L_2(\pi,\mathds{H})$, $\Pi^\pi(f) = \theta_f^\intercal {\bf \Phi} $ where 
\begin{align}\label{proj}
\theta_f =\argmin_{\theta\in\mathds{R}^d} \sqrt{\int_h \left| f(h) - \theta^\intercal {\bf \Phi}(h)\right|^2 \pi(dh)}.
\end{align}
We also define the following approximate Bellman operator for some finite window policy $\gamma$, using the cost function $\hat{c}_\pi(h,u)$ (see (\ref{cost_window})) and the transition model $\eta_\pi$ (see (\ref{eta_N})) such that for some $f\in L_2(\pi,\mathds{H})$, we write that
\begin{align}\label{appr_bell}
T^\gamma f(h) := \int_\mathds{U}\left(\hat{c}_\pi(h,u) + \beta \int f(h_1)\eta_\pi(dh_1|h,u)\right)\gamma(du|h).
\end{align}
\begin{prop}\label{comp_contr}
The mapping $\Pi^\pi T^\gamma$ is a contraction under the $L_2$ norm, and thus admits a unique fixed point.
\end{prop}
\begin{proof}
For $f,g \in L_2(\pi,\mathds{H})$, we have that
\begin{align*}
\|\Pi^\pi T^\gamma (f) - \Pi^\pi T^\gamma (g)\|_2 \leq  \| T^\gamma (f) -  T^\gamma (g)\|_2 
\end{align*}
as the projection is non-expansive. Using the Jensen's inequality, we then have:
\begin{align*}
&\| T^\gamma (f) -  T^\gamma (g)\|_2\\
& \leq \beta \sqrt{   \int \left(f(h_1) - g(h_1)\right)^2\eta_\pi(dh_1|h,u)\gamma(du|h) \pi(dh) }.
\end{align*}
For any $F\in L_2(\pi,\mathds{H})$ we can write the following:
\begin{align*}
\int F(h_1)\eta_\pi(dh_1|h,u)\gamma(du|h) \pi(dh) = E\left[ {E}[F(H_1)|H,U] \right]
\end{align*}
where the outer expectation is with respect to the invariant measure under the true dynamics, and the inner expectation is with respect to the approximate  transition model $\eta_\pi$. However, from (\ref{consistent}), we know that the expectation of $F(h_1)$ under the approximate model is consistent with the true dynamics when the prior distribution over $X_0$ is given by $\pi$, i.e. the marginal of the invariant measure on the $X$ variable. Hence, we can write that
\begin{align*}
E\left[ {E}[F(H_1)|H,U] \right]= E[F(H)]=\int F(h)\pi(dh)
\end{align*}
where we use the fact that $\pi$ is the stationary measure under the policy $\gamma$. Finally, we write
\begin{align*}
&\| T^\gamma (f) -  T^\gamma (g)\|_2 \leq \beta \sqrt{   \int \left(f(h) - g(h)\right)^2\pi(dh) }
\end{align*}
which ends the proof.
To be more precise, we write the following for $N=2$:
\begin{align*}
&\int F(y_2,y_1,u_1)\gamma(du_1|y_1,y_0,u_0)P^\pi(y_2|y_1,y_0,u_0)\\
&\qquad\qquad\qquad\qquad\qquad\qquad\qquad\qquad\pi(dy_1,dy_0,du_0)\\
&=\int F(y_2,y_1,u_1)\gamma(du_1|y_1,y_0,u_0)O(dy_2|x_2)\\
&\qquad\qquad P^{X_0\sim \pi}(dx_2| y_1,y_0,u_0)\pi(dy_1,dy_0,du_0)\\
&=\int  F(y_2,y_1,u_1)O(dy_2|x_2)\mathcal{T}(dx_2|x_1,u_1)\\
&\qquad \gamma(du_1|y_1,y_0,u_0) P^{X_0\sim \pi}(dx_1| y_1,y_0,u_0)\pi(dy_1,dy_0,du_0)\\
&=\int  F(y_2,y_1,u_1)O(dy_2|x_2)\mathcal{T}(dx_2|x_1,u_1)\\
&\qquad\qquad \pi(du_1,dx_1,dy_1,dy_0,du_0)\\
&=\int  F(y_2,y_1,u_1)\pi(dy_2,dy_1,du_1).
\end{align*}
\end{proof}

\subsection{Convergence of the policy evaluation algorithm}  We consider the following iterations:
\begin{align}\label{main_iter}
\theta_{t+1}= \theta_t - \alpha_t {\bf \Phi}(h_t) \left[\theta_t^\intercal {\bf \Phi}(h_t) - c(X_t,U_t) - \beta \theta^\intercal_t{\bf \Phi}(h_{t+1}) \right]
\end{align}
where $\alpha_t$ represent the learning rates, and where we use a single trajectory of $\{X_t,U_t,Y_t\}_{t}$ under the policy $\gamma$.
\begin{theorem}\label{main_thm}
Under Assumption \ref{erg_assmp}, if the learning rates are such that $\sum_t\alpha_t = \infty$ and $\sum_t \alpha_t^2<\infty$, then the iterations in (\ref{main_iter}) converge to some $\theta^* \in\mathds{R}^d$. Denoting by $V(h):={\theta^*}^\intercal {\bf \Phi}(h)$, $V(h)$ is the fixed point of the mapping $\Pi^\pi T^\gamma $ where the mappings $\Pi^\pi$ and $T^\gamma$ are defined in (\ref{proj}) and (\ref{appr_bell}).
\end{theorem}
In what follows we prove this result. We build on the analysis in \cite{tsitsiklis1997analysis} by adapting it to the partially observed case with finite-memory variables. We state the following crucial result taken from \cite[Theorem 2]{tsitsiklis1997analysis} which is adapted from \cite[Theorem 17, p. 239]{benveniste2012adaptive}:
\begin{prop}\label{vanroy_lem}
Consider the iterations given by:
\begin{equation*}
    \theta_{t+1} = \theta_t + \alpha_t \big( A(X_t) \theta_t + b(X_t) \big)
\end{equation*}
where
\begin{enumerate}
    \item the learning rates $\alpha_t$ is positive, non-increasing, and satisfies $\sum_{t=0}^{\infty} \alpha_t = \infty$ and $\sum_{t=0}^{\infty} \alpha_t^2 < \infty$;
    \item $X_t$ is a Markov process with a unique invariant distribution, and there exists a mapping $g$ from the states of the Markov process to the positive reals, satisfying the remaining conditions. Let ${E}_0[\cdot]$ stand for expectation with respect to this invariant distribution;
    \item $A(\cdot)$ and $b(\cdot)$ are matrix and vector valued functions, respectively, for which $A = {E}_0[A(X_t)]$ and $b = {E}_0[b(X_t)]$ are well defined and finite;
    \item the matrix $A$ is negative definite;
    \item there exist constants $C$ and $q$ such that for all $X$
    \begin{equation*}
        \sum_{t=0}^{\infty} \|E[A(X_t) \mid X_0 = X] - A\| \leq C(1 + g^q(X))
    \end{equation*}
    and
    \begin{equation*}
        \sum_{t=0}^{\infty} \|E[b(X_t) \mid X_0 = X] - b\| \leq C(1 + g^q(X));
    \end{equation*}
    \item for any $q > 1$ there exists a constant $\mu_q$ such that for all $X, t$
    \begin{equation*}
        E[g^q(X_t) \mid X_0 = X] \leq \mu_q (1 + g^q(X)).
    \end{equation*}
\end{enumerate}
Then, $\theta_t$ converges to $\theta^*$, with probability one, where $\theta^*$ is the unique vector that satisfies $A\theta^* + b = 0$.
\end{prop}
Before the proof of Theorem \ref{main_thm}, we prove a lemma.
\begin{lemma}\label{helpful_lem}
\begin{align*}
(\theta-\theta^*)^\intercal E\left[ {\bf \Phi}(h)\left[ c(X,U)+\beta \theta^\intercal{\bf \Phi}(H_1) - \theta^\intercal{\bf \Phi}(H)\right] \right]<0
\end{align*}
for any $\theta\neq\theta^*$ where $\theta^*$ corresponds to the fixed point of the operator $\Pi^\pi T^\gamma$, that is ${\theta^*}^\intercal{\bf \Phi}(h)$, where $\theta^*$ is unique since $\Phi^i$'s are assumed to be linearly independent. Furthermore, the expectation is with respect to the unique invariant measure of $(h_t,x_t,u_t)$ under the policy $\gamma$.
\end{lemma}
\begin{proof}
For notation convenience, we denote by $T_\theta(X,U,H_1):= c(X,U)+\beta \theta^\intercal{\bf \Phi}(H_1)$. We further use the notation $\Pi^\pi(T_\theta(X,U,H_1))$ for $\theta^\intercal {\bf \Phi}(H)$ that minimizes $E[\left|T_\theta(X,U,H_1) - \theta^\intercal{\bf \Phi}(H)\right|^2]$. Note that it is equivalent to minimize 
\begin{align*}
&E\left[\left| \hat{c}_\pi(H,U) + \beta \int \theta^\intercal {\bf \Phi}(h_1)  \eta_\pi(dh_1|H,U)- \theta^\intercal{\bf \Phi}(H)\right|^2\right]\\
&=E\left[\left| T^\gamma ( \theta^\intercal{\bf \Phi}(H)  )- \theta^\intercal{\bf \Phi}(H)\right|^2\right].
\end{align*}
Thus, we have that $\Pi^\pi(T_\theta(X,U,H_1))=\Pi^\pi T^\gamma(\theta^\intercal{\bf \Phi}(H)  )$. 

The first order conditions imply that $E[{\bf \Phi}(H)[ T_\theta(X,U,H_1) - \Pi^\pi(T_\theta(X,U,H_1)) ] ]=0$. Then,  by adding and subtracting $\Pi^\pi(T_\theta(X,U,H_1))$:
\begin{align*}
&(\theta-\theta^*)^\intercal E[ {\bf \Phi}(H)[ \left(T_\theta(X,U,H_1) - \Pi^\pi(T_\theta(X,U,H_1))\right)\\
&\qquad\qquad\qquad +\left(   \Pi^\pi(T_\theta(X,U,H_1))-\theta^\intercal{\bf \Phi}(H)\right)] ]\\
& = (\theta-\theta^*)^\intercal E[ {\bf \Phi}(H)\left(   \Pi^\pi(T_\theta(X,U,H_1))-\theta^\intercal{\bf \Phi}(H)\right)] .
\end{align*}
In what follows, we use the equality $\Pi^\pi(T_\theta(X,U,H_1))=\Pi^\pi T^\gamma(\theta^\intercal{\bf \Phi}(H)  )$, and we add and subtract ${\theta^*}^\intercal{\bf \Phi}(H) =\Pi^\pi T^\gamma({\theta^*}^\intercal{\bf \Phi}(H)  )$ to use the contraction property of the composition operator $\Pi^\pi T^\gamma$ (see Proposition \ref{comp_contr}):
\begin{align*}
&(\theta-\theta^*)^\intercal E[ {\bf \Phi}(H)\left(   \Pi^\pi(T_\theta(X,U,H_1))-\theta^\intercal{\bf \Phi}(H)\right)] \\
&= (\theta-\theta^*)^\intercal E[ {\bf \Phi}(H) (\Pi^\pi T^\gamma(\theta^\intercal{\bf \Phi}(H))-  {\theta^*}^\intercal{\bf \Phi}(H)  ) ]\\
& +  (\theta-\theta^*)^\intercal E[ {\bf \Phi}(H) ( {\theta^*}^\intercal{\bf \Phi}(H) -  \theta^\intercal{\bf \Phi}(H)  )]\\
&\leq  \| (\theta-\theta^*)^\intercal  {\bf \Phi}(h)\|_2 \| \Pi^\pi T^\gamma(\theta^\intercal{\bf \Phi}(h))-  {\theta^*}^\intercal{\bf \Phi}(h)\|_2\\
& - \| (\theta-\theta^*)^\intercal  {\bf \Phi}(h)\|^2_2\\
& \leq  (\beta-1) \| (\theta-\theta^*)^\intercal  {\bf \Phi}(h)\|^2_2 <0
\end{align*}
where we used the Cauchy-Schwarz inequality, and the $L_2$ norm is with respect to the invariant measure $\pi$. The last step follows from the uniqueness of $\theta^*$.
\end{proof}

\begin{proof}[Proof of Theorem \ref{main_thm}]
We use Proposition \ref{vanroy_lem} with 
\begin{align*}
&A(H_t,H_{t+1})= \beta{\bf \Phi}(H_t){\bf \Phi}^\intercal(H_{t+1})- {\bf \Phi}(H_t){\bf \Phi}^\intercal(H_{t})\\
&b(H_t,X_t,U_t)={\bf \Phi}(H_t)c(X_t,U_t).
\end{align*}
The matrices $A$ and $b$ are defined under the invariant measure $\pi$ of the  joint process $(h_t,x_t,u_t)$. 
 The assumptions 2,3,5,6 follow from the ergodicity assumption on the joint process $(h_t,x_t,u_t)$. For assumption 5, we write the following for some $\theta\in\mathds{R}^d$ and for $\theta^*$ which corresponds to the fixed point of $\Pi^\pi T^\gamma$:
 \begin{align*}
 &A (\theta-\theta^*)=E\left[\beta {\bf \Phi}(H){\bf \Phi}^\intercal(H_{1})- {\bf \Phi}(H){\bf \Phi}^\intercal(H)  \right] (\theta-\theta^*)\\
 &= E\left[{\bf \Phi}(h)\left( c(X,U)+\beta {\bf \Phi}^\intercal(H_{1})\theta- {\bf \Phi}^\intercal(H)\theta \right) \right]\\
 & \quad - E\left[{\bf \Phi}(H)\left( c(X,U)+\beta {\bf \Phi}^\intercal(H_{1})\theta^*- {\bf \Phi}^\intercal(H)\theta^* \right) \right]\\
 &=E\left[{\bf \Phi}(H)\left( c(X,U)+\beta {\bf \Phi}^\intercal(H_{1})\theta- {\bf \Phi}^\intercal(H)\theta \right) \right]
 \end{align*}
 where the last step follows from the fact that ${\bf \Phi}^\intercal(H)\theta^*$ is the fixed point of the operator $\Pi^\pi T^\gamma$. Moreover, it is also the closest point on the span of ${\bf \Phi}$ to  $c(X,U)+\beta {\bf \Phi}^\intercal(H_{1})\theta^*$ in $L_2$ under the invariant measure $\pi$ of $(h,x,u)$, and thus the error term is orthogonal to the span of ${\bf \Phi}$ functions.  Together with Lemma \ref{helpful_lem}, this show that 
 \begin{align*}
 (\theta-\theta^*)A(\theta-\theta^*)<0
 \end{align*}
 for all $\theta\neq\theta^*$, and thus $\theta_t$ converges to some $\theta'$ that satisfies $A\theta'+b=0$, which implies that
 \begin{align*}
 E\left[{\bf \Phi}(H)\left( c(X,U)+\beta {\bf \Phi}^\intercal(H_{1})\theta'- {\bf \Phi}^\intercal(H)\theta' \right) \right]=0
 \end{align*}
 then as argued earlier, $\theta'$ also satisfies:
 \begin{align*}
 &E\bigg[{\bf \Phi}(H)\bigg( \hat{c}_\pi(H,U)+\beta \int{\bf \Phi}^\intercal(h_{1})\theta' \eta(dh_1|H,U)\\
 &\qquad\qquad\qquad- {\bf \Phi}^\intercal(h)\theta' \bigg) \bigg]=0
 \end{align*}
 which in turn implies that ${\bf \Phi}^\intercal(h)\theta' $ is the fixed point of the operator $\Pi^\pi T^\gamma$. Since the fixed point is unique, we have that $\theta'=\theta^*$ which completes the proof.
\end{proof}

\subsection{Error bounds for the learned model}
In the previous section, we observed that  using the iterations (\ref{main_iter}), one can learn the fixed point of the operator $\Pi^\pi T^\gamma$ (see (\ref{proj}) and (\ref{appr_bell})). In the following, we compare the learned value function ${\theta^*}^\intercal {\bf \Phi}(h)$ with the fixed point of the operator $T^\gamma$. We note that the fixed point of the operator $T^\gamma$ is the value function of the finite-memory policy $\gamma$ for the approximate model constructed in Section \ref{app_section} which we denote by $J_\beta^N(h,\gamma)$. However, this is not the value of the finite-memory policy in the original partially observed environment.
\begin{prop}\label{l2_bound}
Under the invariant measure $\pi$ of the joint process $(h_t,x_t,u_t)$ with the policy $\gamma$, we have that 
\begin{align*}
&\|J_\beta^N(h,\gamma) - {\theta^*}^\intercal {\bf \Phi}(h)\|_2\\
&\qquad\qquad\qquad \leq\frac{1}{1-\beta}   \|J_\beta^N(h,\gamma) - \Pi^\pi (J_\beta^N(h,\gamma) ) \|_2
\end{align*}
\end{prop}
\begin{proof}
We start with the following bound
\begin{align*}
&\|J_\beta^N(h,\gamma) - {\theta^*}^\intercal {\bf \Phi}(h)\|_2\leq \|J_\beta^N(h,\gamma) - \Pi^\pi T^\gamma(J_\beta^N(h,\gamma) ) \|_2\\
&\qquad\qquad \qquad +  \|\Pi^\pi T^\gamma(J_\beta^N(h,\gamma) )-  {\theta^*}^\intercal {\bf \Phi}(h)\|_2\\
&\leq  \|J_\beta^N(h,\gamma) - \Pi^\pi (J_\beta^N(h,\gamma) ) \|_2 + \beta \|J_\beta^N(h,\gamma) - {\theta^*}^\intercal {\bf \Phi}(h)\|_2
\end{align*}
For the first term, since $J_\beta^N(h,\gamma)$ is the fixed point of the operator $T^\gamma$ (under the uniform norm), we have that $ \Pi^\pi T^\gamma(J_\beta^N(h,\gamma) ) = \Pi^\pi J_\beta^N(h,\gamma)$. For the second term, we use the fact that ${\theta^*}^\intercal {\bf \Phi}(h)$ is the fixed point of $\Pi^\pi T^\gamma$ which is a contraction under the $L_2$ norm. Combining the terms concludes the proof.
\end{proof}
The upper bound is related the projection error of the value function $J_\beta^N(h,\gamma)$ onto the span of ${\bf \Phi}$ under the $L_2$ norm of  the stationary measure $\pi$ with the policy $\gamma$. However, in general, there is no guarantee that the initial distribution of the finite-memory variables is consistent with this stationary measure $\pi$. In the following, we derive an  upper bound on the uniform norm difference for near-linear value functions:
\begin{assumption}\label{near_lin}
We assume that there exists some $\hat{\theta}$ and some constant $\lambda<\infty$ such that 
\begin{align*}
\| J_\beta^N(h,\gamma) - \hat{\theta}^\intercal {\bf \Phi}(h)\|_\infty \leq \lambda. 
\end{align*}
\end{assumption}

\begin{prop}\label{uni_bound}
Under Assumption \ref{near_lin}, we have that 
\begin{align*}
&\|J_\beta^N(h,\gamma) - {\theta^*}^\intercal {\bf \Phi}(h)\|_\infty \leq \lambda \left(1+  \frac{2-\beta}{1-\beta}\sqrt{\frac{d}{\sigma_{\min}}}\right)
\end{align*}
where $\theta^*$ is the learned parameter with the iterations in (\ref{main_iter}). Furthermore,  $\sigma_{\min}$ is the minimum eigenvalue of the matrix $E[{\bf \Phi}(h) {\bf \Phi}^\intercal(h)  ] $ when $h$ is distributed with the invariant measure $\pi$.
\end{prop}
\begin{proof}
We begin by adding and subtracting $\hat{\theta}^\intercal {\bf \Phi}(h)$:
\begin{align*}
&\|J_\beta^N(h,\gamma) - {\theta^*}^\intercal {\bf \Phi}(h)\|_\infty\\
& \leq \| J_\beta^N(h,\gamma) - \hat{\theta}^\intercal {\bf \Phi}(h)\|_\infty+ \| \hat{\theta}^\intercal {\bf \Phi}(h) -  {\theta^*}^\intercal {\bf \Phi}(h)\|_\infty.
\end{align*}
The first term is bounded by $\lambda$ by assumption. We analyze the second term under the $L_2$ norm:
\begin{align*}
 &\| \hat{\theta}^\intercal {\bf \Phi}(h) -  {\theta^*}^\intercal {\bf \Phi}(h)\|_2 \\
 &\leq   \| \hat{\theta}^\intercal {\bf \Phi}(h) - J_\beta^N(h,\gamma)\|_2 + \|J_\beta^N(h,\gamma) - {\theta^*}^\intercal {\bf \Phi}(h)\|_2\\
 &\leq \lambda + \frac{1}{1-\beta}   \|J_\beta^N(h,\gamma) - \Pi^\pi (J_\beta^N(h,\gamma) ) \|_2\\
 &\leq \frac{2-\beta}{1-\beta}\lambda.
\end{align*}
For the second inequality, we used Proposition \ref{l2_bound}. Furthermore, by Assumption \ref{near_lin}, the $L_2$ distance between $J_\beta^N(h,\gamma)$ and  $\hat{\theta}^\intercal {\bf \Phi}(h)$ is also bounded $\lambda$ as we work under probability measures. For the last inequality,  we use the fact that since $ \Pi^\pi (J_\beta^N(h,\gamma) ) $ is the projection of  $J_\beta^N(h,\gamma)$ under the $L_2$ norm of $\pi$, then it achieves the minimum $L_2$ distance to $J_\beta^N(h,\gamma)$, and thus it must achieve an error bound less than $\lambda$ that $\hat{\theta}$ achieves. 

On the other hand, we have that 
\begin{align*}
&\| \hat{\theta}^\intercal {\bf \Phi}(h) -  {\theta^*}^\intercal {\bf \Phi}(h)\|^2_2\\
&= (\theta^* - \hat{\theta}) E[{\bf \Phi}(h) {\bf \Phi}^\intercal(h)  ] (\theta^*-\hat{\theta}) \geq  \|\theta^* - \hat{\theta}\|^2_2 \sigma_{\min}
\end{align*}
where $\sigma_{\min}$ is the minimum eigenvalue of the matrix $E[{\bf \Phi}(h) {\bf \Phi}^\intercal(h)  ] $ when $h$ is distributed with the invariant measure $\pi$. Note that the $2$ norm for the $\theta$ vectors is the standard $2$ norm and not to be confused with the $L_2$ norm under $\pi$ over the functions. Combining what we have so far, we can write
\begin{align*}
 \|\theta^* - \hat{\theta}\|_2 \leq \frac{2-\beta}{1-\beta}\frac{\lambda}{\sqrt{\sigma_{\min}}}.
\end{align*}
Going back to the initial term, for any $h$, we have that 
\begin{align*}
&|J_\beta^N(h,\gamma) - {\theta^*}^\intercal {\bf \Phi}(h)|\\
& \leq | J_\beta^N(h,\gamma) - \hat{\theta}^\intercal {\bf \Phi}(h)|+ | \hat{\theta}^\intercal {\bf \Phi}(h) -  {\theta^*}^\intercal {\bf \Phi}(h)|\\
&\leq \lambda +  \|\theta^* - \hat{\theta}\|_2  \| {\bf \Phi}(h)\|_2\leq \lambda +  \frac{2-\beta}{1-\beta}\frac{\lambda\sqrt{d}}{\sqrt{\sigma_{\min}}}
\end{align*}
where we used the assumption that $\|\Phi^i\|_\infty\leq 1$ for all basis functions. Hence, the proof is complete.
\end{proof}

The next result is the main result of this section, and provides an error upper-bound for the learned value function with respect to the true value of the finite-memory policy in the original environment.
\begin{theorem}
We assume that the unobserved state initiates at time $-N$ according to some $\mu_{-N}\in \P(\mathds{X})$, and the finite-memory policy $\gamma$ starts acting at time $t=0$. We denote by $h_0$, the finite-memory variables from time $t=-N$ to $t=0$. For ${z}_0=(\mu_{-N},h_0)$, with a policy $\hat{\gamma}$ acting on the first $N$ steps, we have that
\begin{align*}
&E_{\mu_{-N}}^{\hat{\gamma}}\left[\left|J_\beta({z}_0,{\gamma}) -{\theta^*}^\intercal {\bf \Phi}(h_0)\right|\right]\\
&\leq  \frac{\|c\|_\infty }{(1-\beta)}\sum_{t=0}^\infty\beta^tL_t + \lambda\left(1+ \frac{2-\beta}{1-\beta}  \sqrt{\frac{d}{\sigma_{\min}}  }\right)
\end{align*}
 where the expectation is with respect to the random realizations of the initial finite-memory variables $h_0$. 
\end{theorem}
\begin{proof}
The proof is an application of Proposition \ref{cont_bound} and Proposition \ref{uni_bound}.
\begin{align*}
&E_{\mu_{-N}}^{\hat{\gamma}}\left[\left|J_\beta({z}_0,{\gamma}) -{\theta^*}^\intercal {\bf \Phi}(h_0)\right|\right]\\
&\leq  E_{\mu_{-N}}^{\hat{\gamma}}\left[\left|J_\beta({z}_0,{\gamma}) -  J_\beta^N(h,\gamma) \right|\right] \\
&\qquad+ E_{\mu_{-N}}^{\hat{\gamma}}\left[\left| J_\beta^N(h,\gamma) - {\theta^*}^\intercal {\bf \Phi}(h_0)  \right|\right]
\end{align*}
the first term is bounded by Proposition \ref{cont_bound} and the second term is bounded by Proposition \ref{uni_bound}.
\end{proof}

\section{Learning Approximate Optimal Q-Values }
In this section, we shift our focus to approximately learning the optimal Q-values using finite-memory and linear function approximations. 
We first study the case of general basis functions with a somewhat restrictive assumption. In this section, we extend our basis functions by using: $\{\phi^i(h,u)\}_{i=1}^d$ where  $\phi^i(h,u):\left(\mathds{Y}\times\mathds{U}\right)^N\to \mathds{R}$.
We assume that $
\|\phi^i\|_\infty\leq 1$
for all $i=1,\dots,d$.

 Consider the following iterations, where we denote by $V_t(h)=\min_v \theta_t^\intercal {\bf \Phi}(h,v)$:
\begin{align}\label{opt_iter}
\theta_{t+1}= \theta_t - \alpha_t {\bf \Phi}(H_t,U_t) &\big[\theta_t^\intercal {\bf \Phi}(H_t,U_t) \nonumber\\
&- c(X_t,U_t) - \beta V_t(H_{t+1}) \big]
\end{align}
where the actions are chosen under some time invariant finite-memory exploration policy $\gamma:\mathds{Y}^N\times\mathds{U}^{N-1}\to\mathds{U}$.
We assume the ergodicity condition under the exploration policy, meaning that Assumption \ref{erg_assmp} holds for the exploration policy $\gamma$.
We note that the ergodicity assumption is not as restrictive as in the previous section, since randomized exploration policies are more common in general. Therefore, imposing $\gamma(du|h)\geq \lambda(\cdot)$ is a design choice rather than a restriction in this section.
We denote by
\begin{align}\label{sigma_gamma}
\Sigma_\gamma:=E\left[{\bf \Phi}(h,u){\bf \Phi}^\intercal(h,u)\right]
\end{align}
where $(h,u)$ is distributed according  to the invariant measure of the process $(h_t,x_t,u_t)$ under the exploration policy $\gamma$. We also denote by $\gamma_\theta(h)=\argmin_u \theta^\intercal {\bf \Phi}(h,u) $ the greedy policy for the parameter $\theta$. We define
\begin{align}\label{sigma_star}
\Sigma_\theta:=E\left[{\bf \Phi}(h,\gamma_\theta(h)){\bf \Phi}^\intercal(h,\gamma_\theta(h))\right]
\end{align}
where $h$ is distributed according to the invariant measure of $(h_t,x_t,u_t)$.
We further define the Bellman operator under the greedy action selection such that 
\begin{align}\label{greedy_bell}
T f(h,u) := \hat{c}_\pi(h,u) + \beta \int \inf_vf(h_1,v)\eta_\pi(dh_1|h,u)
\end{align}
where $\pi$ is the invariant measure under the exploration policy. Similar to the projection map in (\ref{proj}), $\Pi^\pi$, in this section denotes the projection map over the span of the basis functions $\{\phi^i(h,u)\}_{i=1}^d$.

For the convergence of the algorithm, we impose the following assumption:
\begin{assumption}\label{rest_assmp}
For all $\theta\in\mathds{R}^d$
\begin{align*}
 \beta^2 \Sigma_\theta< \Sigma_\gamma.
\end{align*}
\end{assumption}
We note that this assumption is parallel to the assumption used in \cite{melo2008analysis}, and indicates that for large $\beta$, the greedy policy and the exploration policy are close to each other, which can be rather restrictive in practice.
\begin{theorem}
Under Assumption \ref{erg_assmp} for the exploration policy and Assumption \ref{rest_assmp}, if the learning rates are such that $\sum_t\alpha_t = \infty$ and $\sum_t \alpha_t^2<\infty$, then the iterations in (\ref{opt_iter}) converge to some $\theta^* \in\mathds{R}^d$. Denoting by $Q(h,u):={\theta^*}^\intercal {\bf \Phi}(h,u)$, $Q(h,u)$ is the fixed point of the mapping $\Pi^\pi T $ where the mappings $\Pi^\pi$ and $T$ are defined in (\ref{proj}) and (\ref{greedy_bell}).
\end{theorem}
\begin{proof}[Sketch of the proof]
The proof is very similar to the proof of Theorem \ref{main_thm} and uses Proposition \ref{vanroy_lem}. The key difference is to show that the composition operator $\Pi^\pi T $ is a contraction in $L_2(\pi)$. We set $\theta_f,\theta_g$, such that $f(h,u)=\theta_f^\intercal {\bf \Phi}(h,u)$.
 We can follow identical steps as in the proof of Lemma \ref{comp_contr} up to the following step
\begin{align*}
&\| T (f) -  T (g)\|^2_2 \leq \beta^2 {   \int \left(\min_vf(h,v) - \min_vg(h,v)\right)^2 \pi(dh) }.
\end{align*}
We can show that 
$\left|\min_vf(h,v) - \min_vg(h,v)\right|\leq \max_\theta \left|f(x,\gamma_\theta(x))-g(x,\gamma_\theta(x))\right|$. Denoting the maximum achieving $\theta$ by $\bar{\theta}$:
\begin{align*}
&\beta^2 {   \int \left(\min_vf(h,v) - \min_vg(h,v)\right)^2 \pi(dh) } \\
&\leq\beta^2 (\theta_f-\theta_g)^\intercal \int {\bf \Phi}(h,\gamma_{\bar{\theta}}(h))  {\bf \Phi}^\intercal(h,\gamma_{\bar{\theta}}(h)) \pi(dh)(\theta_f-\theta_g)\\
&=  \beta^2(\theta_f-\theta_g)^\intercal \Sigma_{\bar{\theta}}  (\theta_f-\theta_g)\\
&<(\theta_f-\theta_g)^\intercal \Sigma_{\gamma}  (\theta_f-\theta_g) = \|f-g\|_2^2
\end{align*}
where we used Assumption \ref{rest_assmp} for the last inequality.
\end{proof}

\subsection{Convergence under discretization}
For the analysis so far, we have worked with the $L_2$ norm over the functions of the finite-memory variables.  We have observed that the discrepancy between the exploration policy and the greedy policy within the Bellman operator makes the contraction analysis non-trivial for optimal Q-value estimation. In this section, we discuss some special cases for which the projection mapping does not expand the supremum norm of the functions. Accordingly, one can directly work with the uniform norm $\|\cdot\|_\infty$ for the contraction analysis.

Consider the invariant measure $\pi$ of the joint process $(h_t,x_t,u_t)$ under the exploration policy $\gamma$, and the basis functions $\{\phi^i(h,u)\}_{i=1}^d$. Recall the projection mapping $\Pi^\pi$; in this section, we study special cases where $\|\Pi^\pi(f)\|_\infty \leq \|f\|_\infty$ for any $f\in L_2(\pi)$. 

A special case occurs when $f$ belongs to the span of the basis functions $\{\phi^i(h,u)\}_{i=1}^d$. 
In particular, if the cost function $\hat{c}_\pi$ and the kernel $\eta_\pi$ are perfectly linear with respect to the basis functions, we then have that $\Pi^\pi(T(f)) = T(f)$ and thus
\begin{align*}
\||\Pi^\pi(T(f))-\Pi^\pi(T(f))\|_\infty = \|T(f)-T(g)\|_\infty\leq \beta\|f-g\|_\infty
\end{align*}
which can be used to show the convergence of the algorithm presented in (\ref{opt_iter}) without Assumption \ref{rest_assmp}. 

Another important case is linear approximation via discretization of the observation and the action spaces. For a weak Feller belief MDP (\cite{FeKaZa12, KSYWeakFellerSysCont}), \cite[Theorem 3.16]{SaLiYuSpringer}  has established near optimality of finite action policies. If $\mathds{U}$ is compact, a finite collection of action sets can be constructed, with arbitrary approximation error. Accordingly, we will assume that the action spaces are finite in the following. Let $\{B_i\}_{i=1}^{M}$ be disjoint subsets of $\mathds{Y}$ such that $\cup_{i=1}^{M}B_i=\mathds{Y}$. This discretization then implies a discretization on the finite-memory and action variables $(h,u)\in(\mathds{Y\times U})^N$.
We denote by $\{A_i\}_{i=1}^{(M\times |\mathds{U}|)^N}$ for the resulting discretization bins of the joint $(h,u)\in(\mathds{Y\times U})^N$ variable.
We define the following basis functions 
\begin{align*}
\phi^i(h,u)=\mathds{1}_{A_i}(h,u), \text{ for all } i=1,\dots, (M\times |\mathds{U}|)^N
\end{align*}
where $\mathds{1}_{A_i}(h,u)$ is the indicator function of the set $A_i$. Note that the projection map $\Pi^\pi$ is such that $\Pi^\pi(f)(h,u)=\theta^\intercal {\bf \Phi}(h,u)$, where $\theta=\Sigma^{-1}_\gamma E_{\pi}\left[ {\bf \Phi}(h,u)f(h,u)\right]$
for the invariant measure $\pi$ under the exploration policy $\gamma$ where $\Sigma_\gamma$ is defined in (\ref{sigma_gamma}). For the particular case of discretization, the basis functions $\phi^i$ are perfectly orthonormal and only one of them is equal to 1, and the rest are 0 for any input $(h,u)$. We then have that $
\Sigma^{-1}_\gamma(i,i)=\frac{1}{\pi(A_i)}
$
and it has $0$ entries for the non-diagonal elements. Thus, we can show that for some $(h,u)\in A_i$ 
\begin{align*}
\Pi^\pi(f)(h,u) &= \frac{\int_{A_i} f(h',u')\pi(dh',du')}{\pi(A_i)}\\
&=\int_{A_i} f(h',u')\pi_i(dh',du')\leq \sup_{h,u\in A_i} |f(h,u)|
\end{align*}
where $\pi_i(dh,du)$ is a probability measure normalized over $A_i$. Therefore, we have that $\|\Pi^\pi(f)\|_\infty\leq \|f\|_\infty$, and in particular,  the composition operator $\Pi^\pi T$ is a contraction under the supremum norm.  Due to these structural properties of the projection mapping based on the discretization, we can derive sharper error analysis results. The following is adapted from   \cite{devran2025near} based on the results in this paper:
\begin{assumption}\label{main_assmp}
\begin{itemize}
\item {$\mathds{Y}\subset\mathds{R}^n$} is compact.
\item $O(dy|x)=g(x,y)\lambda(dy)$, and $g(y,x)$ is Lipschitz in $y$, such that $|g(x,y)-g(x,y')|\leq\alpha_{\mathds{Y}} \|y-y'\|$ for every $y,y'\in\mathds{Y}$ and $x\in\mathds{X}$ for some $\alpha_{\mathds{Y}}<\infty$.
\item Stage-wise cost function $c(x,u)$ is bounded such that $\sup_{x,u}c(x,u)=\|c\|_\infty<\infty$.
\end{itemize}
\end{assumption}

\begin{theorem}
\begin{itemize}
\item
Under Assumption \ref{erg_assmp} for the exploration policy, if the learning rates are such that $\sum_t\alpha_t = \infty$ and $\sum_t \alpha_t^2<\infty$, then the iterations in (\ref{opt_iter}) converge to some $\theta^* \in\mathds{R}^d$.
\item
Suppose Assumption \ref{main_assmp} holds. Consider the learned policy $\gamma^N$, which satisfies $\gamma^N(h)=\argmin_u {\theta^*}^\intercal {\bf \Phi}(h,u)$. We assume that the unobserved state initiates at time $-N$ according to some $\mu_{-N}\in \P(\mathds{X})$, and the learned finite-memory policy $\gamma$ starts acting at time $t=0$. We denote by $h_0$, the finite-memory variables from time $t=-N$ to $t=0$. For ${z}_0=(\mu_{-N},h_0)$, with a policy $\hat{\gamma}$ acting on the first $N$ steps, we have that
\begin{align*}
&E_{\mu_{-N}}^{\hat{\gamma}}\left[\left|J_\beta({z}_0,{\gamma^N}) - J_\beta^*(z_0)\right|\right]\\
&\leq  \frac{2\|c\|_\infty }{(1-\beta)}\sum_{t=0}^\infty\beta^t\hat{L}_t + \frac{\beta}{(1-\beta)^2}\|c\|_\infty \alpha_{\mathds{Y}} L_\mathds{Y} \end{align*}
 where the expectation is with respect to the random realizations of the initial finite-memory variables $h_0$
where
\begin{align*}
 &L_{\mathds{Y}}:= \max_{i}\sup_{y,y'\in B_i}\|y-y'\|,\\
\hat{L}&_t:=\sup_{\hat{\gamma}\in\hat{\Gamma}}E_{\mu}^{\hat{\gamma}}\bigg[\|P^{\pi_t^-}(X_{t+N}\in\cdot|\hat{Y}_{[t,t+N]},U_{[t,t+N-1]})\nonumber\\
&\qquad\qquad-P^{\pi^*}(X_{t+N}\in\cdot|\hat{Y}_{[t,t+N]},U_{[t,t+N-1]})\|_{TV}\bigg]
\end{align*}
such that the filter stability term $\hat{L}_t$ is with respect to the discretized observations and $\alpha_{\mathds{Y}}$ is the Lipschitz constant of the density function $g$ of the channel $O$.
\end{itemize}
\end{theorem}

\bibliographystyle{plain}

\bibliography{AliBibliography,references_acc,SerdarBibliography_acc,references}

\end{document}